\documentclass[a4paper,11pt]{amsart}
\usepackage[a4paper,twoside,centering]{geometry}
\usepackage[arrow,matrix,curve]{xy}

\title{2-CY-tilted algebras that are not Jacobian}
\author{Sefi Ladkani}
\address{%
Institut des Hautes \'{E}tudes Scientifiques \\
Le Bois Marie, 35, route de Chartres \\
91440 Bures-sur-Yvette, France}
\email{sefil@ihes.fr}

\DeclareMathOperator{\ch}{char}
\DeclareMathOperator{\coker}{Coker}
\DeclareMathOperator{\End}{End}
\DeclareMathOperator{\hh}{H}
\DeclareMathOperator{\HC}{HC}
\DeclareMathOperator{\HH}{HH}
\DeclareMathOperator{\Hom}{Hom}
\DeclareMathOperator{\id}{id}
\DeclareMathOperator{\img}{Im}
\DeclareMathOperator{\modf}{mod}
\DeclareMathOperator{\rank}{rank}
\DeclareMathOperator{\stmod}{\underline{mod}}
\DeclareMathOperator{\tensor}{\wh{\otimes}}

\newcommand{\cC}{\mathcal{C}}
\newcommand{\cD}{\mathcal{D}}
\newcommand{\cH}{\mathcal{H}}
\newcommand{\dcH}{\cD^b(\cH)}
\newcommand{\gL}{\Lambda}
\newcommand{\bQ}{\mathbb{Q}}
\newcommand{\bSigma}{\bar{\Sigma}}
\newcommand{\bZ}{\mathbb{Z}}
\newcommand{\wh}{\widehat}
\newcommand{\wt}{\widetilde}
\newcommand{\vphi}{\varphi}

\newtheorem{lemma}{Lemma}[section]
\newtheorem{propp}[lemma]{Proposition}
\newtheorem{prop}{Proposition}
\newtheorem*{cor*}{Corollary}

\theoremstyle{definition}
\newtheorem{defn}[lemma]{Definition}
\newtheorem{example}[lemma]{Example}
\newtheorem*{defn*}{Definition}
\newtheorem*{example*}{Example}
\newtheorem{remark}{Remark}

\numberwithin{equation}{section}

\begin{document}

\begin{abstract}
Over any field of positive characteristic we construct 2-CY-tilted
algebras that are not Jacobian algebras of quivers with potentials.
As a remedy, we propose an extension of the notion of a potential,
called hyperpotential, that allows to prove that certain
algebras defined over fields of positive characteristic are
2-CY-tilted even if they do not arise from potentials.

In another direction, we compute the fractionally Calabi-Yau
dimensions of certain orbit categories of fractionally CY triangulated
categories. As an application, we construct a cluster category of
type $G_2$.
\end{abstract}

\maketitle

\section{Introduction}

A 2-CY-tilted algebra is an endomorphism algebra of a cluster-tilting object in
a 2-Calabi-Yau triangulated category.
There are close connections between 2-CY-tilted algebras and 
Jacobian algebras of quivers with potentials as introduced by Derksen, Weyman
and Zelevinsky~\cite{DWZ08}.
On the one hand, already in~\cite{DWZ08} it is shown that cluster-tilted
algebras of Dynkin type, which are particular kind of 2-CY-tilted algebras,
are Jacobian algebras. Later, Buan, Iyama, Reiten and Smith have shown
in~\cite{BIRS11} that all cluster-tilted algebras, and more generally
the 2-CY-tilted algebras arising from cluster categories associated
in~\cite{BIRS09} to words in Coxeter groups are Jacobian.
Moreover, they have shown that under some conditions
the notions of mutation of cluster-tilting objects in a 2-CY
category and mutation of quivers with potentials are compatible.

On the other hand, by the work of Amiot~\cite{Amiot09},
any finite-dimensional Jacobian algebra is 2-CY-tilted.
It is therefore natural to ask whether any 2-CY-tilted algebra is a
Jacobian algebra of a quiver with potential~\cite[Question~2.20]{Amiot11}.
The purpose of this note is twofold.
First, we provide a negative answer to this question
over any field of positive characteristic. Our examples are given by certain
self-injective Nakayama algebras which are also known as truncated cycle
algebras. Second, we show that it is actually possible to slightly
extend the notion of a potential in order to exclude this kind of examples.
Let us explain the motivation behind such extension.

Since 2-CY-tilted algebras have some remarkable homological and
structural properties~\cite{KellerReiten07}, it is of interest to know
that certain finite-dimensional algebras defined in a uniform way over
all fields (e.g.\ as quivers with relations ``over $\bZ$'') are
2-CY-tilted. Often this is done by ``integrating'' the defining relations
to give a potential so that the algebra could be seen as a Jacobian
algebra. However, there are cases where such ``integration'' is only
possible provided we restrict the characteristic of the field.

Consider for example the algebra $\gL_K = K[x]/(x^{n-1})$ over a field
$K$ for some $n>2$,
which could be described as a quiver with one vertex, one loop $x$
and a relation $x^{n-1}$. As long as the characteristic of $K$
does not divide $n$, this algebra is Jacobian (take the potential $x^n$) and
hence 2-CY-tilted. However, the AR-quiver of $\gL_K$ and the fact that
it is symmetric do not depend on the field $K$, so one would like to
say that $\gL_K$ is 2-CY-tilted regardless of the characteristic of $K$.
Another example of the same kind is given by the remark of
Ringel~\cite[\S14]{Ringel11} that barbell algebras with two loops are
2-CY-tilted, provided that one assumes
that the characteristic of the ground field is not 3.

As some of the constructions involving a quiver with potential rely only on
its cyclic derivatives (see e.g.\ the definitions of the Ginzburg
dg-algebra or the Jacobian algebra), our idea is to replace these
derivatives with arbitrary elements and to concentrate on the required
conditions that these elements have to satisfy in order for such
constructions to make sense. This will avoid the need
to ``integrate'' relations into one potential and will allow to prove in
a characteristic-free manner that certain algebras are 2-CY-tilted.

\subsection{Hyperpotentials}

Recall that the construction of Amiot~\cite{Amiot09} starts with a 
dg-algebra $\Gamma$ which is concentrated in non-positive degrees,
homologically smooth, bimodule 3-CY and whose 0-th cohomology
$\hh^0(\Gamma)$ is finite-dimensional, and produces a 
2-CY triangulated category with a cluster-tilting object
whose endomorphism algebra is $\hh^0(\Gamma)$.
This construction is applied to a quiver with potential $(Q,W)$ over
a field $K$
by considering its Ginzburg dg-algebra defined in~\cite{Ginzburg06}.
Keller proves in~\cite{Keller11} that the Ginzburg dg-algebra has the
required properties by showing that it is quasi-isomorphic
to the deformed 3-Calabi-Yau completion of the 
path algebra $KQ$ by an element in
$\HH_1(KQ)$
which is the image of the potential $W$ under Connes' map
$\HC_0(KQ) \to \HH_1(KQ)$.

These considerations raise the possibility of working from the outset with
elements in $\HH_1$ (and not in $\HC_0$) and motivate the following
definition.
Indeed, Ginzburg's original definition in~\cite[\S5]{Ginzburg06} starts
with a cyclic 1-form satisfying certain conditions, which is not necessarily
a differential of a potential.

\begin{defn*}
Let $K$ be a commutative ring and let $Q$ be a quiver.
Denote by  $Q_0$, $Q_1$ the sets of vertices and arrows of $Q$
and by $A=\wh{KQ}$ the completed path algebra of $Q$ over $K$
(i.e.\ its elements of are infinite $K$-linear combinations of paths in $Q$).
For any $i \in Q_0$ let $e_i \in A$ be the idempotent corresponding
to the path of length 0 starting at $i$.

A \emph{hyperpotential} on $Q$ is a collection of elements
$(\rho_\alpha)_{\alpha \in Q_1}$ in $A$ indexed by the arrows of $Q$ satisfying the
following conditions:
\begin{enumerate}
\renewcommand{\theenumi}{\roman{enumi}}
\item
If $\alpha \colon i \to j$ then $\rho_\alpha \in e_j A e_i$.
In other words, $\rho_\alpha$ is a (possibly infinite) linear combination
of paths starting at $j$ and ending at $i$.

\item \label{it:comm}
$\sum_{\alpha \in Q_1} [\alpha, \rho_\alpha] = 0$ in $A$.
\end{enumerate}
\end{defn*}

Hyperpotentials represent elements in $\HH_1(A)$, and
any potential $W \in \HC_0(A)$ gives rise to a hyperpotential by considering
its cyclic derivatives $(\partial_\alpha W)_{\alpha \in Q_1}$. Conversely, when the
ring $K$ contains $\bQ$, any hyperpotential arises in this way, so there
is nothing new. However, when $K$ does not contain $\bQ$ (e.g.\ when it is
a field of positive characteristic) there are hyperpotentials that do
not arise from potentials but nevertheless one would like to attach to them
suitable 3-CY and 2-CY categories.

In order to do that, one defines the \emph{Ginzburg dg-algebra}
$\Gamma(Q,(\rho_\alpha))$ of a hyperpotential $(Q,(\rho_\alpha))$
in the usual way, see~\cite[\S5.2]{Ginzburg06}
and~\cite[\S2.6]{KellerYang11};
Let $\wt{Q}$ be the graded quiver whose set of vertices is $Q_0$ and
whose arrows are the arrows of $Q$ (in degree $0$)
together with an arrow
$\alpha^* \colon j \to i$ of degree $-1$
for each arrow $\alpha \colon i \to j$ in $Q_1$
and a loop $t_i$ of degree $-2$ at each vertex $i \in Q_0$.
As a graded algebra, $\Gamma(Q,(\rho_\alpha))$ is the completion of
the graded path algebra $K\wt{Q}$ with respect to path length
(so that each graded piece consists of the infinite linear combinations
of paths of a given degree).
Its differential is defined as the
continuous linear map homogeneous of degree 1
which satisfies the Leibniz rule and whose values on the generators are
given by
\begin{align*}
d(\alpha)=0 &,& d(\alpha^*)=\rho_\alpha &,&
d(t_i) = e_i \Bigl(\sum_{\beta \in Q_1} [\beta,\beta^*]\Bigr) e_i
\end{align*}
for each $i \in Q_0$ and $\alpha \in Q_1$.
Note that the condition $d^2=0$ is equivalent
to the condition~\eqref{it:comm} in the definition of hyperpotential.
The \emph{Jacobian algebra} of a hyperpotential $(Q,(\rho_\alpha))$
is defined as the 0-th cohomology of its Ginzburg dg-algebra.
Equivalently, it is the quotient of $\wh{KQ}$ by the closure of the ideal
generated by the elements $\rho_\alpha$ for $\alpha \in Q_1$.

By following the proof of~\cite[Theorem~6.3]{Keller11} by Keller we
deduce:
\begin{prop}
The Ginzburg dg-algebra of a hyperpotential is (topologically)
homologically smooth and 3-CY.
\end{prop}
Note that~\cite{Keller11} treats the non-completed version of the 
Ginzburg algebra. The corresponding statement for the completed version
appears in~\cite[Theorem~A.17]{KellerYang11}.

When $K$ is a field, the results of Amiot~\cite{Amiot09}
(see also~\cite[\S A.20]{KellerYang11} for the completed case)
imply the following.
\begin{cor*}
If the Jacobian algebra of a hyperpotential is finite-dimensional,
then it is 2-CY-tilted.
\end{cor*}

It may be possible to develop a theory of mutations for hyperpotentials
as done for potentials by Derksen, Weyman and Zelevinsky in~\cite{DWZ08}.
However, since our original motivation was to show that certain algebras
are 2-CY-tilted, we will not pursue this direction here.

In Section~\ref{sec:equiv} we define the notions of \emph{right equivalence}
and \emph{weak right equivalence} for hyperpotentials and
characterize them as arising from isomorphisms of the corresponding
Ginzburg dg-algebras having certain prescribed properties.
In particular, the categorical constructions do not distinguish between
weakly equivalent hyperpotentials.
We show that (weakly) right equivalent potentials
as defined in~\cite{DWZ08} and~\cite{GLS13} are so also when considered
as hyperpotentials.

\subsection{The examples}

We demonstrate the usefulness of the notion of a hyperpotential by
presenting the following class of examples providing a negative answer
to Question~2.20 in~\cite{Amiot11}.
Let $K$ be a field and let $m,e \geq 1$ such that $me \geq 3$.
Consider the $K$-algebra $\gL_{m,e}$ given as the path algebra of the quiver
$Q_m$ which is a cycle with $m$ vertices and arrows
$\alpha_1, \dots, \alpha_m$ (as shown in Figure~\ref{fig:Qm})
modulo the ideal generated by all paths of length $me-1$. 
It is well known that $\gL_{m,e}$ is a finite-dimensional self-injective
Nakayama algebra over $K$. Denote by $\ch K$ the characteristic of $K$.

\begin{prop} \label{p:QP}
Let $K$ be a field and let $m, e \geq 1$ such that $me \geq 3$.
\begin{enumerate}
\renewcommand{\theenumi}{\alph{enumi}}
\item \label{it:QP:hyper}
$\gL_{m,e}$ is the Jacobian algebra of the hyperpotential 
$(\rho_{\alpha_i})_{i=1}^m$ on $Q_m$ given by
\begin{align*}
\rho_{\alpha_i} = 
\alpha_{i+1} \ldots \alpha_{i-1}
(\alpha_i \alpha_{i+1} \ldots \alpha_{i-1})^{e-1}
&&
(1 \leq i \leq m)
\end{align*}
and hence it is always 2-CY-tilted.

\item \label{it:QP:W}
Let $W$ be any potential on the quiver $Q_m$. Then the Jacobian algebra
of $(Q_m,W)$ is either the completed path algebra $\widehat{KQ_m}$
or the algebra $\gL_{m,d}$ for some $d \geq 1$ not divisible by $\ch K$.

\item \label{it:QP:Wme}
Conversely, if $\ch K$ does not divide $e$,
then $\gL_{m,e}$ is the Jacobian algebra of the quiver with potential
$(Q_m,W_{m,e})$, where $W_{m,e} = (\alpha_1 \alpha_2 \ldots \alpha_m)^e$.

\item \label{it:QP:none}
If $\ch K$ divides $e$, then $\gL_{m,e}$ is not a Jacobian algebra of
a quiver with potential.
\end{enumerate}
\end{prop}

\begin{figure}
\[
\xymatrix@=1pc{
&& {\bullet} \ar[drr]^{\alpha_1} \\
{\bullet} \ar[urr]^{\alpha_m} && && {\bullet} \ar[dd]^{\alpha_2} \\ \\
{} \ar[uu]^{\alpha_{m-1}} \ar@{.}@/_1pc/[rrrr] && && {}
}
\]
\caption{The quiver $Q_m$ which is a cycle on $m$ vertices.}
\label{fig:Qm}
\end{figure}
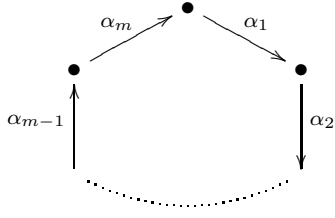

\subsection{Orbit categories}

Another approach to the construction of 2-CY triangulated categories
is via the machinery of triangulated orbit categories developed
by Keller~\cite{Keller05}. Based on this approach, the next statement
provides a construction, which is independent on the characteristic of
$K$, of an ambient 2-CY category for the algebras $\gL_{m,e}$.
Denote by $D_n$ an orientation of the Dynkin diagram of type $D$ with $n$
vertices (where for $n=3$ we use the convention that $D_3=A_3$).

\begin{prop} \label{p:cat}
Let $m,e$ be as in Proposition~\ref{p:QP} and
assume in addition that $m$ is even or that $e$ is odd.
\begin{enumerate}
\renewcommand{\theenumi}{\alph{enumi}}
\item
There is an auto-equivalence $F$ of the bounded derived
category $\cD^b(\modf KD_{me})$ such that the orbit category
\[
\cC_{m,e} = \cD^b(\modf KD_{me})/F
\]
is a 2-CY triangulated category with a cluster-tilting object whose
endomorphism algebra is $\gL_{m,e}$.

\item
The shape of the AR-quiver of
$\cC_{m,e}$ is $\bZ D_{me}/\langle (\phi \tau)^m \rangle$ where
$\phi$ is the automorphism of order 2 of the Dynkin diagram
underlying $D_{me}$.
\end{enumerate}
\end{prop}

The proof of Proposition~\ref{p:cat} relies on the next observation
dealing more generally with
orbit categories of fractionally Calabi-Yau categories,
whose proof uses a result of Keller~\cite{Keller05} on
triangulated orbit categories together with calculations of
Calabi-Yau dimensions by Dugas~\cite[\S9]{Dugas12}.
Recall that a triangulated $K$-linear category $\cC$ with suspension
functor $\Sigma$ and a Serre functor $S$ is \emph{fractionally Calabi-Yau
of dimension $(d,e)$} for some $(d,e) \in \bZ^2 \setminus \{(0,0)\}$
(or $(d,e)$-CY for short) if $S^e \simeq \Sigma^d$.
Observe that the set of pairs $(d,e) \in \bZ^2$ satisfying
$S^e \simeq \Sigma^d$ forms a lattice.

\begin{prop} \label{p:CY}
Let $\cH$ be a hereditary $K$-linear category such that its derived
category $\dcH$ is equivalent to $\cD^b(\modf A)$ for some
finite-dimensional $K$-algebra $A$. Assume that $\dcH$ is
$(d_1,e_1)$-CY for some $(d_1,e_1) \in \bZ^2 \setminus \{(0,0)\}$.

Consider the orbit category $\cC = \dcH / F$ where 
$F=S^{e_2} \Sigma^{-d_2}$ for some $(d_2,e_2) \in \bZ^2$.
Then the following conditions are equivalent:
\begin{enumerate}
\renewcommand{\theenumi}{\roman{enumi}}
\item
$e_1 d_2 - e_2 d_1 \neq 0$.

\item
$\rank L = 2$, where $L$ is the lattice
$L = \bZ(d_1,e_1) + \bZ(d_2,e_2) \subseteq \bZ^2$.

\item
The category $\cC$ is $\Hom$-finite.
\end{enumerate}
Moreover, if any of these conditions holds then the orbit category $\cC$
is triangulated. If $d_2$ is even or $e_2$ is odd
and in addition $d_1-d_2$ is even or $e_1-e_2$ is odd, 
then $\cC$ is $(d,e)$-CY for any $(d,e) \in L$.
In particular, if $(d,1) \in L$ then $\cC$ is $d$-Calabi-Yau.
\end{prop}

\begin{cor*}
Under the assumptions of the proposition, for any rational number
$r \in \bQ$ there is a fraction $(d,e)$ such that the orbit
category $\cC$ is $(d,e)$-CY with $d/e=r$.
\end{cor*}

\subsection{Remarks}
We make a few remarks on these observations.

\begin{remark}
The assumptions on the parity of $d_1$, $d_2$, $e_1$ and $e_2$
in Proposition~\ref{p:CY}
are needed in order to apply the results of Dugas~\cite{Dugas12}, who
observed that otherwise there may be delicate sign issues.
They could be dropped at the expense of replacing the
lattice $L$ by its sublattice $2L \subseteq L$ of index 4.
\end{remark}

\begin{remark}
By taking $\cH$ to be the category of representations of a Dynkin quiver,
Proposition~\ref{p:CY} and its corollary apply in particular to the
cluster categories~\cite{BMRRT06}, higher cluster categories~\cite{Thomas07}
and repetitive higher cluster categories~\cite{Lamberti14}
associated with Dynkin quivers. Moreover, by Amiot's
classification~\cite{Amiot07} of triangulated categories with finitely many
indecomposables, it applies also to many stable categories of self-injective
algebras of finite representation type, see for example~\cite{Dugas12}.
Many authors~\cite{Dugas12,ErdmannSkowronksi06,HolmJorgensen13}
have considered only the intersection of the lattice $L$ with the ray
$\{(n,1) \,:\, n \geq 0\}$.

One could take $\cH$ to be any other hereditary 
category whose derived category is fractionally Calabi-Yau. Over an
algebraically closed field, such categories were classified by
van Roosmalen~\cite{vanRoosmalen12}. In particular,
the proposition and its corollary apply also to 
the tubular cluster categories studied by Barot and
Geiss~\cite{BarotGeiss12}
which are special cases of the cluster categories associated
with canonical algebras~\cite{BKL10}. Here, $\cH$ is the category of
sheaves over a weighted projective line in the sense of Geigle and
Lenzing~\cite{GeigleLenzing87}.
\end{remark}

\begin{example*}
Let $\cC$ be the cluster category of tubular type $(2,2,2,2;\lambda)$
for $\lambda \neq 0,1$.
In this case, the hereditary category $\cH$ is $(2,2)$-CY and
$(d_2,e_2)=(2,1)$, hence from $(0,1)=(2,2)-(2,1)$ we see that
$\cC$ is not only 2-CY, but 0-CY as well.
It follows that the endomorphism algebra of any object in $\cC$ is
symmetric, and in particular all the cluster-tilted algebras of tubular
type $(2,2,2,2;\lambda)$, whose description as quivers with potentials
can be found in~\cite[Figure~1]{GKO13}, are symmetric.
\end{example*}

\begin{remark}
By using Amiot's description~\cite{Amiot07} of triangulated categories with
finitely many indecomposables and the classification
in~\cite[Appendix~A]{BIKR08} of the possible shapes of the AR-quivers of
such 2-CY categories with cluster-tilting objects, Bertani-{\O}kland and
Oppermann have classified all representation-finite 2-CY-tilted algebras
arising from standard algebraic 2-CY triangulated categories over an
algebraically closed field~\cite[Theorem~5.7]{BertaniOppermann11}.
Proposition~\ref{p:cat} can be seen as a special case of their classification,
and in the notations of~\cite{BIKR08}, the category $\cC_{m,e}$ corresponds
to type $D_{me}$ with generator $(m,\bar{m}) \in \bZ \times \bZ/2\bZ$.

However, our proofs will not rely on~\cite[Appendix~A]{BIKR08}.
As an illustration of our methods, let us give the following example of
a 2-CY-tilted algebra of finite representation type which seems not to
appear in~\cite{BertaniOppermann11,BIKR08}.
\end{remark}

\begin{example*}[Cluster category of type $G_2$]
The algebra $\gL$ given as the quiver
\[
\xymatrix{
{\bullet} \ar[r] & {\bullet} \ar@(ur,dr)[]^{\beta}
}
\]
with the relation $\beta^3=0$ is of finite representation type and its
AR-quiver is shown in the lecture notes of Gabriel~\cite[Fig.~19]{Gabriel80}.
It has 30 vertices arranged in a cylinder, and by inserting 2 additional
vertices one gets the translation quiver $\bZ E_8/\langle \tau^4 \rangle$.

Indeed, there is an auto-equivalence of $\cD^b(\modf KE_8)$ such that
the corresponding orbit category $\cC_{G_2}$
is 2-CY triangulated with a cluster-tilting
object whose endomorphism algebra is isomorphic to $\gL$.
The AR-quiver of $\cC_{G_2}$ is $\bZ E_8/\langle \tau^4 \rangle$ and
the exchange graph of cluster-tilting objects is an octagon.
The algebra $\gL$ is Jacobian precisely when the characteristic of $K$ is
not 2, and in this case a potential is $\beta^4$
(note that $\beta^3$ is always a hyperpotential).
The category $\cC_{G_2}$ models a cluster algebra of type $G_2$;
we give more details in Section~\ref{sec:G2}.
\end{example*}

\begin{remark}
When $e=1$, the category $\cC_{m,1}$ in Proposition~\ref{p:cat}
is the cluster category, as introduced in~\cite{BMRRT06},
of the Dynkin quiver $D_m$ and it is well known that $\gL_{m,1}$ is a
cluster-tilted algebra of type $D_m$.
In particular, it appears in Ringel's classification of the
self-injective cluster-tilted algebras~\cite{Ringel08}.
Observe that $\gL_{m,1}$ is a Jacobian algebra in any characteristic.
\end{remark}

\begin{remark}
The quiver $Q_m$ belongs to the mutation class of the Dynkin quiver $D_m$
which is acyclic, thus it has a unique non-degenerate potential up to
right equivalence. It follows that the potential $W_{m,e}$ 
in Proposition~\ref{p:QP} is non-degenerate
if and only if $e=1$. Therefore the categories $\cC_{m,e}$ for $e>1$ will not
properly model the corresponding cluster algebra.
\end{remark}

\begin{remark}
Self-injective Jacobian algebras were studied by Herschend and
Iyama~\cite{HerschendIyama11}. However, observe that unless $e=1$,
the quiver with potential $(Q_m,W_{m,e})$ has no cut in their sense,
so that the algebras $\gL_{m,e}$ for $e>1$ do not arise as 3-preprojective
algebras of 2-representation-finite algebras of global dimension~2.
\end{remark}

\begin{remark}
If $\gL$ is a self-injective 2-CY-tilted algebra with $m$ simple modules
arising from an ambient 2-CY category $\cC$, we can also consider its
stable module category $\stmod \gL$ which is triangulated. There are
stabilization functors
\[
\cC \to \modf \gL \to \stmod \gL
\]
where the left functor was considered by Keller and
Reiten~\cite{KellerReiten07}, who also showed that $\stmod \gL$ is
3-CY. At each stage, the AR-quiver of the next category is obtained from that
of the previous one by deleting $m$ vertices (corresponding to indecomposable
summands of a suitable cluster-tilting object in $\cC$, or to the
indecomposable projectives in $\modf \gL$, respectively).
In our case we get a sequence
\[
\cC_{m,e} \to \modf \gL_{m,e} \to \stmod \gL_{m,e}
\]
where $\cC_{m,e}$ has $m^2e$ indecomposables and the AR-quiver of
$\stmod \gL_{m,e}$ has the shape of the translation quiver
$\bZ A_{me-2}/\langle \tau^m \rangle$.

The self-injective algebras of finite representation type whose
stable module categories are higher cluster categories have been
classified by Holm and Jorgensen~\cite{HolmJorgensen13}.
In particular, $\stmod \gL_{m,1}$ is the 3-cluster category of type
$A_{m-2}$.
\end{remark}

\begin{remark}
The algebras $\gL_{m,e}$ are symmetric precisely when $m \leq 2$ (i.e.\
when the quiver is a loop or a 2-cycle).
In this case, they have been shown to be 2-CY-tilted (at least
in characteristic zero) by Burban, Iyama, Keller and Reiten~\cite{BIKR08}.
The ambient 2-CY triangulated categories considered there are the stable
categories of maximal Cohen-Macaulay modules over simple hypersurface
singularities of odd dimension.
\end{remark}

\section{The proofs}

\subsection{Hochschild and cyclic homology for completed path algebras}

In this section we explain why hyperpotentials are the elements of the
first (continuous) Hochschild homology of the completed path algebra.
Let $K$ be a commutative ring.
Let $Q$ be a finite quiver and denote by
$Q_0$ its set of vertices and by $Q_1$ its set of arrows.
Let $A=\wh{KQ}$ be the completed path algebra of $Q$ over $K$.
For any $i \in Q_0$, let $e_i \in A$ be the idempotent corresponding to
the trivial path at $i$, and let $A_+$ be the subspace of (infinite)
linear combinations of non-trivial paths (it is a two-sided ideal in $A$).

The path algebra $KQ$ is a tensor algebra over the commutative ring
$R = \bigoplus_{i \in Q_0} Ke_i$ of the projective $R$-bimodule
$\bigoplus_{\alpha \in Q_1} K\alpha$.
The Hochschild and cyclic homology of tensor algebras were computed by
Loday and Quillen~\cite[\S5]{LodayQuillen84}. Let us recall the result
in our case.

Let $\sigma \colon A \to A$ be the continuous linear map defined by
$\sigma(e_i)=e_i$ for $i \in Q_0$ and by
$\sigma(\alpha_1 \ldots \alpha_n) = \alpha_n \alpha_1 \ldots \alpha_{n-1}$
for a path $\alpha_1 \ldots \alpha_n$.
Obviously, $\sigma$ vanishes on paths that are not cycles.
Then
\begin{align*}
\HH_0(A) &= \coker (\id - \sigma) = A_{\sigma}\\
\HH_1(A) &= \ker ((\id - \sigma)_{\big| A_+}) = A_+^{\sigma}
\end{align*}
and $\HH_n(A)=0$ for $n \geq 2$.
The elements of the space $A_\sigma$ of $\sigma$-coinvariants are infinite
linear combinations of cycles modulo rotation (each cycle can be rotated
independently).
Indeed, if $x=\sum c_n x_n$ is a sum of cycles $x_n$ and $i_n \geq 0$ are
arbitrary,
then $x$ and $\sum c_n \sigma^{i_n} x_n$ differ by $(\id - \sigma)(y)$
for $y = \sum c_n y_n$ with
$y_n = (\id + \sigma + \dots + \sigma^{i_n-1})(x_n)$
(here, $y_n=0$ if $i_n=0$).
Thus, the space of potentials is precisely $(A_+)_{\sigma}$.

The space $A_+^{\sigma}$ of $\sigma$-invariants consists of infinite linear
combinations of non-trivial cycles that are invariant under rotation.
If $\alpha \in Q_1$ and $\rho$ is such that $\alpha \rho$ is a cycle,
then $\sigma(\alpha \rho) = \rho \alpha$. Now,
any linear combination of cycles in $A_+$ can be written in a unique way as
$\sum_{\alpha \in Q_1} \alpha \rho_\alpha$. Since
$\sigma(\sum_{\alpha \in Q_1}
\alpha \rho_\alpha) = \sum_{\alpha \in Q_1} \rho_\alpha \alpha$,
we see that $\sum_{\alpha \in Q_1} \alpha \rho_\alpha \in A_+^{\sigma}$
if and only if $(\rho_\alpha)_{\alpha \in Q_1}$ is a hyperpotential.
In this way we get an identification between hyperpotentials and
elements in $\HH_1(A)$.

From the Connes' exact sequence
\[
\dots \to \HH_n(A) \to \HC_n(A) \to \HC_{n-2}(A) \to \HH_{n-1}(A) \to \dots
\]
we see that $\HC_0(A) \simeq \HH_0(A)$ and
\[
0 \to \HC_2(A) \to \HC_0(A) \xrightarrow{B} \HH_1(A) \to \HC_1(A) \to 0.
\]
The Connes' map $B$ is induced by the norm map $N \colon A \to A$ which is
the continuous linear map defined by
$N(e_i)=0$ for $i \in Q_0$ and by
$N(\alpha_1 \dots \alpha_n) =
\sum_{j=1}^n \alpha_j \alpha_{j+1} \dots \alpha_{j-1}$
for a path $\alpha_1 \dots \alpha_n$.
Comparing this with the definition of the cyclic derivative $\partial_\alpha$
with respect to an arrow $\alpha \in Q_1$,
\[
\partial_\alpha(\alpha_1 \dots \alpha_n) = \sum_{j \,:\, \alpha_j = \alpha}
\alpha_{j+1} \dots \alpha_{j-1}
\]
we see that an element $W \in \HC_0(A)$ is sent by $B$ to
$\sum_{\alpha \in Q_1} \alpha \partial_\alpha W$, and hence any potential $W$
can be regarded as the hyperpotential $(\partial_\alpha W)_{\alpha \in Q_1}$.

Let us write an explicit resolution which is a special case of the
treatment in~\cite[\S6.1]{Keller11}.
For an arrow $\alpha \colon i \to j$, set $s(\alpha)=i$ and $t(\alpha)=j$.
Denote by $\tensor$ the completed tensor product over $K$ and let
$A^e = A^{op} \tensor A$. As an $A^e$-module, $A$ has a projective resolution
\[
0 \to \bigoplus_{\alpha \in Q_1} A e_{s(\alpha)} \tensor e_{t(\alpha)} A \to
\bigoplus_{i \in Q_0} A e_i \tensor e_i A \to A \to 0
\]
where the right map sends an element
$p \otimes q$ to $pq$ and the left map sends an element
$p \otimes q$ in the $\alpha$ component to
$p \alpha \otimes q - p \otimes \alpha q$.
Indeed, this complex is contractible as a complex of right $A$-modules via
the homotopy defined by the continuous maps sending a path $p$ starting at
$i$ to $e_i \otimes p$ and an element
$\alpha_1 \dots \alpha_n \otimes q$ to
$\big(\sum_{j \,:\, \alpha_j=\alpha} \alpha_1 \dots \alpha_{j-1} \otimes
\alpha_{j+1} \dots \alpha_n q \big)_{\alpha \in Q_1}$.

Applying the isomorphism $(A \tensor A) \otimes_{A^e} A \simeq A$ given by
$(a \otimes b) \otimes x \mapsto bxa$, we get the following complex which
computes Hochschild homology
\[
0 \to \bigoplus_{\alpha \in Q_1} e_{t(\alpha)} A e_{s(\alpha)} \to
\bigoplus_{i \in Q_0} e_i A e_i \to 0
\]
where the middle map sends $(x_\alpha)_{\alpha \in Q_1}$ to
$\sum_{\alpha \in Q_1} [x_\alpha,\alpha]$, from which we identify $\HH_1(A)$
with the space of hyperpotentials.

\subsection{Equivalences of hyperpotentials}
\label{sec:equiv}

Let $Q$, $Q'$ be two quivers on the same set of vertices, and
let $A=\wh{KQ}$, $A'=\wh{KQ'}$ be their completed path algebras.
We will denote by $e'_i$ the idempotent in $A'$ corresponding
to the vertex $i \in Q_0$. Consider a continuous homomorphism
of algebras $\vphi \colon A \to A'$ satisfying
$\vphi(e_i)=e'_i$ for all $i \in Q_0$. It induces a map
$\vphi_* \colon \HH_1(A) \to \HH_1(A')$ which we now explicitly
compute in terms of hyperpotentials.

For an arrow $\alpha \in Q_1$, let
$\Delta_\alpha \colon A \to A \tensor A$ be the continuous
(double) derivation taking the values $\Delta_\alpha(e_i)=0$
for $i \in Q_0$ and
\[
\Delta_\alpha(\beta) = \begin{cases}
e_{s(\alpha)} \otimes e_{t(\alpha)} & \text{if $\beta=\alpha$,} \\
0 & \text{otherwise}
\end{cases}
\]
for $\beta \in Q_1$. By induction we get
\[
\Delta_\alpha(\alpha_1 \dots \alpha_n) = 
\sum_{j \,:\, \alpha_j=\alpha} \alpha_1 \dots \alpha_{j-1} \otimes
\alpha_{j+1} \dots \alpha_n
\]
for any path $\alpha_1 \dots \alpha_n$.
The isomorphism $(A \tensor A) \otimes_{A^e} A \simeq A$
is induced by the operation $\diamond$ whose values
on (topological) generators are
$(a \otimes b) \diamond x = bxa$ for $a, b, x \in A$,
and with these notations we have
$\partial_\alpha x = \Delta_\alpha(x) \diamond 1$
for every $x \in A$ and $\alpha \in Q_1$. Observe that always
$\Delta_\alpha(x) \diamond y \in e_{t(\alpha)}Ae_{s(\alpha)}$.

\begin{lemma} \label{l:comm}
Let $x \in A_+$, $y \in A$. Then
\[
\sum_{\alpha \in Q_1} [\alpha, \Delta_\alpha(x) \diamond y] = [x,y]
\]
\end{lemma}
\begin{proof}
By continuity and linearity we may assume that $x=\alpha_1 \dots \alpha_n$
is a non-trivial path, so that
\begin{align*}
\sum_{\alpha \in Q_1} [\alpha, \Delta_\alpha(x) \diamond y] &=
\sum_{j=1}^n \alpha_j \alpha_{j+1} \dots \alpha_n y \alpha_1 \dots \alpha_{j-1}
-
\sum_{j=1}^n \alpha_{j+1} \dots \alpha_n y \alpha_1 \dots \alpha_{j-1} \alpha_j
\\
&= \alpha_1 \dots \alpha_n y - y \alpha_1 \dots \alpha_n = [x,y] .
\end{align*}
\end{proof}
In particular,
by taking $y=1$ we get the well-known identity $\sum_{\alpha \in Q_1}
[\alpha, \partial_\alpha x] = 0$ which justifies why we can
consider a potential $x$ as a hyperpotential
$(\partial_\alpha x)_{\alpha \in Q_1}$.

Now let $(\rho_\alpha)_{\alpha \in Q_1}$ be a hyperpotential on $Q$ and
let $\vphi \colon A \to A'$ as in the beginning of this section.
We define a hyperpotential $(\rho'_\beta)_{\beta \in Q'_1}$ by
\begin{align} \label{e:phihyper}
\rho'_\beta = \sum_{\alpha \in Q_1} \Delta_\beta(\vphi(\alpha)) \diamond
\vphi(\rho_\alpha) && (\beta \in Q'_1)
\end{align}
This is indeed well-defined, since by Lemma~\ref{l:comm} we have
\begin{align*}
\sum_{\beta \in Q'_1} [\beta, \rho'_\beta] &=
\sum_{\beta \in Q'_1} \Bigl[ \beta,
\sum_{\alpha \in Q_1} \Delta_\beta(\vphi(\alpha)) \diamond \vphi(\rho_\alpha)
\Bigr] =
\sum_{\alpha \in Q_1} \sum_{\beta \in Q'_1} [\beta,
\Delta_\beta (\vphi(\alpha)) \diamond \vphi(\rho_\alpha)] \\
&= \sum_{\alpha \in Q_1} [\vphi(\alpha), \vphi(\rho_\alpha)] =
\vphi \Bigl( \sum_{\alpha \in Q_1} [\alpha, \rho_\alpha] \Bigr) = 0 .
\end{align*}

\begin{lemma} \label{l:phihyper}
We have $\vphi_*((\rho_\alpha)_{\alpha \in Q_1}) = (\rho'_\beta)_{\beta \in Q'_1}$.
\end{lemma}

For the proof, we start by considering the Hochschild complex
\[
\dots \to A \tensor A \tensor A \xrightarrow{d_2}
A \tensor A \xrightarrow{d_1} A
\]
with $d_1(a \otimes b) = ab-ba$
and $d_2(a \otimes b \otimes c) = ab \otimes c - a \otimes bc + ca \otimes b$.
Denote by $\sim$ the equivalence relation on $A \tensor A$ defined
by $\img d_2$ (i.e.\
two elements are equivalent if their difference lies in $\img d_2$).

\begin{lemma}
Any element in $A \tensor A$ is equivalent to an element of
the form $\sum_{\alpha \in Q_1} \rho_\alpha \otimes \alpha$ with
$\rho_\alpha \in e_{t(\alpha)} A e_{s(\alpha)}$ for each
$\alpha \in Q_1$.
\end{lemma}
\begin{proof}
Consider an element $y \otimes x \in A \tensor A$ for
some $x,y \in A$. From
$e_i \otimes e_j = e_i e_i \otimes e_j \sim
e_i \otimes e_i e_j - e_j e_i \otimes e_i$ we see that
$e_i \otimes e_j \sim 0$ for all $i,j \in Q_0$. Moreover,
$y \otimes e_i = 1 \cdot y \otimes e_i \sim 1 \otimes ye_i - e_i \otimes y$
and hence we may assume that $x \in A_+$.

If $\alpha_1 \dots \alpha_n$ is a path, then
$y \otimes \alpha_1 \dots \alpha_n \sim
y \alpha_1 \dots \alpha_{n-1} \otimes \alpha_n
+ \alpha_n y \otimes \alpha_1 \dots \alpha_{n-1}$,
so by induction
\[
y \otimes \alpha_1 \dots \alpha_n \sim
\sum_{j=1}^n \alpha_{j+1} \dots \alpha_n y \alpha_1 \dots \alpha_{j-1}
\otimes \alpha_j
= \sum_{\alpha \in Q_1} (\Delta_\alpha(\alpha_1 \dots \alpha_n) \diamond y)
\otimes \alpha
\]
and by linearity and continuity
\begin{equation} \label{e:yx}
y \otimes x \sim \sum_{\alpha \in Q_1} (\Delta_\alpha(x) \diamond y) \otimes
\alpha
\end{equation}
for all $x \in A_+$, $y \in A$.
\end{proof}

To complete the proof of Lemma~\ref{l:phihyper}, note that
the hyperpotential $(\rho_\alpha)$ could be seen as the element
$\sum_{\alpha \in Q_1} \rho_\alpha \otimes \alpha \in \ker d_1 \subseteq
A \tensor A$. Applying $\vphi_*$ and using~\eqref{e:yx}, we get
\[
\vphi_* \Bigl(\sum_{\alpha \in Q_1} \rho_\alpha \otimes \alpha \Bigr)
= \sum_{\alpha \in Q_1} \vphi(\rho_\alpha) \otimes \vphi(\alpha)
\sim \sum_{\alpha \in Q_1} \sum_{\beta \in Q'_1}
(\Delta_\beta(\vphi(\alpha)) \diamond \vphi(\rho_\alpha))
\otimes \beta = \sum_{\beta \in Q'_1} \rho'_\beta \otimes \beta .
\]

The notion of right equivalence for potentials was defined in~\cite{DWZ08},
and weak right equivalence was introduced in~\cite{GLS13}. Let us
consider analogous notions for hyperpotentials.

\begin{defn}
Let $Q$ and $Q'$ be two quivers on the same set of vertices.

Two hyperpotentials $(Q, (\rho_\alpha)_{\alpha \in Q_1})$ and
$(Q', (\rho'_\beta)_{\beta \in Q'_1})$ are \emph{right equivalent}
if there exists a continuous isomorphism of algebras
$\vphi \colon \wh{KQ} \to \wh{KQ'}$
with $\vphi(e_i)=e'_i$ for all $i \in Q_0$ such that
$\vphi_*((\rho_\alpha)_{\alpha \in Q_1}) = (\rho'_\beta)_{\beta \in Q'_1}$.

They are \emph{weakly right equivalent} if there exists $c \in K^{\times}$
such that $(Q, (c \rho_\alpha)_{\alpha \in Q_1})$ and
$(Q', (\rho'_\beta)_{\beta \in Q'_1})$ are right equivalent.
\end{defn}

The original definitions of these notions for two quivers with potentials
$(Q,W)$ and $(Q',W')$ could be rephrased in terms of the map
$\vphi_* \colon \HH_0(A) \to \HH_0(A')$, namely
as $\vphi_*(W)=W'$ and $\vphi_*(cW)=W'$, respectively.

\begin{lemma} \label{l:poteq}
Two potentials that are (weakly) right equivalent
are also (weakly) right equivalent as hyperpotentials.
Conversely, if $K$ contains $\bQ$, then two potentials that are
(weakly) right equivalent as hyperpotentials are so also
as potentials.
\end{lemma}
\begin{proof}
Combining the chain rule~\cite[Lemma~3.9]{DWZ08}, Lemma~\ref{l:phihyper}
and Eq.~\eqref{e:phihyper}, we see that
for any $W \in \HH_0(A)$ and
continuous algebra homomorphism $\vphi \colon A \to A'$ such that
$\vphi(e_i)=e'_i$ for all $i \in Q_0$ we have
\[
\vphi_*((\partial_\alpha W)_{\alpha \in Q_1}) =
(\partial_\beta \vphi(W))_{\beta \in Q'_1} ,
\]
hence the first part. For the second part, note that the
kernel of the Connes' map $\HC_0(A') \to \HH_1(A')$
is $\HC_2(A')$, which is spanned by the trivial paths
if $K$ contains~$\bQ$.
\end{proof}

The second part of Lemma~\ref{l:poteq}
is not true in positive characteristic.
\begin{example}
Let $K$ be a field of characteristic $p$ and let $n > p$. 
Consider a quiver with one vertex and one loop $\alpha$
(this example could be extended to cycles of longer length).
Then the potentials $W=\alpha^n$ and $W'=\alpha^n+\alpha^p$ are
not (weakly) right equivalent since any automorphism would map $W$
to a power series in $\alpha$ whose terms have all degree at least
$n$. However, $W$ and $W'$ yield the same hyperpotential, namely,
$n\alpha^{n-1}$, hence as hyperpotentials they are right equivalent.
\end{example}

The next proposition characterizes (weakly) right equivalent
hyperpotentials in terms of
the existence of isomorphisms between their Ginzburg dg-algebras
having certain prescribed values.
It follows that the 3-CY triangulated categories as well as the
generalized cluster categories associated to weakly equivalent
hyperpotentials are equivalent.
For potentials, the ``only if'' direction in part~\eqref{it:G:equiv}
has been shown in~\cite[Lemma~2.9]{KellerYang11}.

\begin{propp} \label{p:equiv}
Let $Q$ and $Q'$ be two quivers on the same set of vertices,
let $(Q,(\rho_\alpha))$ and $(Q',(\rho'_\beta))$ be hyperpotentials
and let $\Gamma$, $\Gamma'$ be the corresponding Ginzburg dg-algebras.
\begin{enumerate}
\renewcommand{\theenumi}{\alph{enumi}}
\item \label{it:G:equiv}
$(Q,(\rho_\alpha))$ and $(Q',(\rho'_\beta))$ are right equivalent if
and only if there exists a continuous isomorphism of dg-algebras
$\Phi \colon \Gamma \to \Gamma'$
such that $\Phi(e_i)=e'_i$ and $\Phi(t_i)=t'_i$ for all $i \in Q_0$.

\item \label{it:G:weak}
$(Q,(\rho_\alpha))$ and $(Q',(\rho'_\beta))$ are weakly right equivalent if
and only if there exist a continuous isomorphism of dg-algebras
$\Phi \colon \Gamma \to \Gamma'$ and an element $c \in K^{\times}$
such that $\Phi(e_i)=e'_i$ and $\Phi(t_i)=c t'_i$ for all $i \in Q_0$.
\end{enumerate}
\end{propp}
\begin{proof}
Let $\Gamma=\Gamma(Q,(\rho_\alpha))$ and $\Gamma'=\Gamma(Q',
(\rho'_\beta))$ be the Ginzburg dg-algebras of two hyperpotentials.
We determine the possible form of a continuous isomorphism
$\Phi \colon \Gamma \to \Gamma'$ such that $\Phi(e_i)=e'_i$
and $\Phi(t_i)=t'_i$ for all $i \in Q_0$. 
First, by looking at the degree 0
part, $\Phi$ induces a continuous isomorphism $\vphi \colon A \to A'$
with values $\vphi(e_i)=e'_i$ and $\vphi(\alpha)=\Phi(\alpha)$
for $i \in Q_0$ and $\alpha \in Q_1$.

Now, the set of elements $\vphi(u) \beta^* \vphi(v)$
where $\beta \in Q'_1$ and $u, v$ are paths in $Q$ such that
$v$ starts at $s(\beta)$ and $u$ ends at $t(\beta)$
forms a (topological) basis of the degree $-1$
part of $\Gamma'$. Hence we can write for $\alpha \in Q_1$
\[
\Phi(\alpha^*) = \sum_{u, v, \beta} b^{\alpha,\beta}_{u,v}
\vphi(u) \beta^* \vphi(v)
\]
for some scalars $b^{\alpha,\beta}_{u,v} \in K$.
Since $\Phi(e_i)=e'_i$ for all $i \in Q_0$, the coefficient
$b^{\alpha,\beta}_{u,v}$ can be non-zero only when $v$ ends at $s(\alpha)$
and $u$ starts at $t(\alpha)$. Thus, we can rewrite this sum
as ranging over all the non-trivial paths $p$ in $Q$ and all their
factorizations $p=v \alpha u$ as
\begin{equation} \label{e:Phi_first}
\Phi(\alpha^*) = \sum_{\beta \in Q'_1} \sum_{p=\alpha_1 \dots \alpha_n}
\sum_{j \,:\, \alpha_j = \alpha} b^{\beta}_{p,j}
\vphi(\alpha_{j+1} \dots \alpha_n) \beta^* \vphi(\alpha_1 \dots \alpha_{j-1})
\end{equation}
where $b^{\beta}_{p,j} \in K$ are some coefficients defined
for the paths $p=\alpha_1 \dots \alpha_n$ and $1 \leq j \leq n$.

We use now our assumption that $\Phi(t_i)=t'_i$ for all $i \in Q_0$ and
the fact that $\Phi$ should commute with the differentials, hence
\[
\sum_{\beta \in Q'_1} [\beta, \beta^*] =
\Phi \Bigl( \sum_{\alpha \in Q_1} [\alpha, \alpha^*] \Bigr)
= \sum_{\alpha \in Q_1} [\vphi(\alpha), \Phi(\alpha^*)] .
\]
Plugging in the expression~\eqref{e:Phi_first} and
comparing coefficients, we deduce that $\sum_p b^\beta_{p,1} = \beta$
and $b^\beta_{p,j+1} = b^{\beta}_{p,j}$ for all $1 \leq j < n$,
so we can rewrite~\eqref{e:Phi_first} as
\begin{equation} \label{e:Phi}
\Phi(\alpha^*) = \sum_{\beta \in Q'_1} \sum_{p=\alpha_1 \dots \alpha_n}
b^{\beta}_{p} \sum_{j \,:\, \alpha_j = \alpha} 
\vphi(\alpha_{j+1} \dots \alpha_n) \beta^* \vphi(\alpha_1 \dots \alpha_{j-1})
\end{equation}
where the coefficients $b^{\beta}_p$ are defined by the equations
$\vphi^{-1}(\beta) = \sum_p b^{\beta}_p p$ for all $\beta \in Q'_1$.

Taking differentials of~\eqref{e:Phi}, noting that
$\Phi(d \alpha^*) = \vphi(\rho_\alpha)$, we get
\begin{align*}
\vphi(\rho_\alpha) &= 
\sum_{\beta \in Q'_1} \sum_{p=\alpha_1 \dots \alpha_n}
b^{\beta}_{p} \sum_{j \,:\, \alpha_j = \alpha} 
\vphi(\alpha_{j+1} \dots \alpha_n) \rho_\beta
\vphi(\alpha_1 \dots \alpha_{j-1}) \\
&=
\sum_{\beta \in Q'_1} \sum_p b^{\beta}_{p}
\vphi\left(\Delta_\alpha(p) \diamond \vphi^{-1}(\rho_\beta)\right) =
\sum_{\beta \in Q'_1}
\vphi\left(\Delta_\alpha(\vphi^{-1}(\beta)) \diamond 
\vphi^{-1}(\rho_\beta)\right)
\end{align*}
or in other words, $(\rho_\alpha)_{\alpha \in Q_1} =
\vphi^{-1}_*((\rho_\beta)_{\beta \in Q'_1})$.

This proves one direction in part~\eqref{it:G:equiv}.
For the other direction, we define $\Phi \colon \Gamma \to \Gamma'$
by specifying its values on $e_i$, $\alpha$, $\alpha^*$ and $t_i$,
namely by setting $\Phi(e_i)=e'_i$, $\Phi(t_i)=t'_i$,
$\Phi(\alpha)=\vphi(\alpha)$ and finally $\Phi(\alpha^*)$ according
to~\eqref{e:Phi}. One then shows that such $\Phi$ is actually an
isomorphism in the same way as in~\cite[Lemma~2.9]{KellerYang11}.

Part~\eqref{it:G:weak} of the proposition now follows from
the next statement which can be verified by direct calculation.
\end{proof}

\begin{lemma}
Let $(Q,(\rho_\alpha)_{\alpha \in Q_1})$ be a hyperpotential and let
$c \in K^{\times}$. Then the continuous map of $K$-algebras
$\Phi \colon \Gamma(Q, (c \rho_\alpha)) \to \Gamma(Q, (\rho_\alpha))$
whose values on the generators are
\begin{align*}
\Phi(e_i) = e_i &,& \Phi(\alpha) = \alpha &,&
\Phi(\alpha^*) = c \alpha^* &,& \Phi(t_i) = c t_i
\end{align*}
for $i \in Q_0$ and $\alpha \in Q_1$, is an isomorphism of dg-algebras.
\end{lemma}

\subsection{Proof of Proposition~\protect{\ref{p:QP}}}

We will freely use some notions from the theory of quivers with potentials.
For details and explanations, we refer the reader to the paper~\cite{DWZ08}.

For part~\eqref{it:QP:hyper}, just note that
\begin{align*}
\sum_{i=1}^n \alpha_i \rho_{\alpha_i} =
\sum_{i=1}^n (\alpha_i \alpha_{i+1} \dots \alpha_{i-1})^e =
\sum_{i=1}^n (\alpha_{i+1} \dots \alpha_{i-1} \alpha_i)^e =
\sum_{i=1}^n \rho_{\alpha_i} \alpha_i
\end{align*}
(indices are taken modulo $n$, i.e.\ $\alpha_{n+1}=\alpha_1$ and
$\alpha_0 = \alpha_n$).

Let $W$ be a potential on $Q_m$. For an arrow $\alpha=\alpha_i$, set
\begin{align*}
\omega_\alpha = \alpha_i \alpha_{i+1} \ldots \alpha_{i-1} &&
\omega'_\alpha = \alpha_{i+1} \dots \alpha_{i-1}
\end{align*}
so that $\omega_\alpha = \alpha \omega'_\alpha$. Since the cycles
$\omega_\alpha^r$ and $\omega_\beta^r$ are cyclically equivalent for 
any two arrows $\alpha, \beta$ and $r \geq 1$, the potential $W$ is
cyclically equivalent to $P(\omega_\alpha)$ for some power series
$P(x) \in K[[x]]$ which is independent on the arrow $\alpha$, hence
we may assume that $W = P(\omega_\alpha)$.
Taking cyclic derivatives, we see that
$\partial_\alpha W = \partial_\alpha P(\omega_\alpha) =
P'(\omega_\alpha) \omega'_\alpha$ for any arrow~$\alpha$.

If $P'(x)=0$, then the Jacobian ideal vanishes
and the Jacobian algebra equals the completed path algebra
$\widehat{KQ_m}$. Otherwise,
\[
P'(x) = a_0 x^{d-1} + a_1 x^d + \ldots
\]
for some $a_0 \neq 0$ and $d \geq 2$. Note that $d$ cannot be divisible
by $\ch K$, as otherwise the derivative of $x^d$ would vanish and so
$a_0=0$, a contradiction.

We deduce that in the Jacobian algebra, for any arrow $\alpha$
\[
\rho_\alpha \mathrel{\mathop:}= \omega_\alpha^{d-1} \omega'_\alpha = -a_0^{-1}
\Big(
\sum_{i=1}^{\infty} \omega_\alpha^i (\omega_\alpha^{d-1} \omega'_\alpha)
\Big)
= -a_0^{-1} \Big(\sum_{i=1}^{\infty} \omega_\alpha^i \Big) \rho_\alpha
\]
hence the path $\rho_\alpha$ of length $md-1$ equals
a linear combination of paths of length at least $m(d+1)-1$.
Since we can repeatedly substitute for $\rho_\alpha$ the expression in the RHS,
we get that for any $N \geq 1$, the element
$\rho_\alpha$ equals a linear combination of paths of length at least $N$, and
hence it vanishes. Since no shorter paths are involved in any relation,
we deduce that
the Jacobian algebra equals $\gL_{m,d}$. This proves part~\eqref{it:QP:W}.
Part~\eqref{it:QP:Wme} follows by taking $P(x)=x^e$, observing that
$P'(x) = ex^{e-1}$ does not vanish by the assumption on the characteristic
of $K$.

To show part~\eqref{it:QP:none}, observe that by the Splitting Theorem
of~\cite{DWZ08}, we may assume that the potential is reduced, and hence
that the quiver equals $Q_m$. In view of part~\eqref{it:QP:W}, the
algebra $\gL_{m,e}$ is not a Jacobian algebra of any potential on $Q_m$.

\subsection{Proof of Proposition~\protect{\ref{p:cat}}}

We recall some facts on $\cD^b(\modf K D_n)$,
the bounded derived category of the path algebra of the Dynkin quiver $D_n$
on $n \geq 3$ vertices.
It has a translation functor $\Sigma$ and a Serre functor $S$,
and they are related by $S^{n-1} \simeq \Sigma^{n-2}$ if $n$ is even and
$S^{2n-2} \simeq \Sigma^{2n-4}$ if $n$ is odd (for this fractionally
Calabi-Yau property, see~\cite{MiyachiYekutieli01}).
The AR-translation on $\cD^b(\modf K D_n)$ is given by $\tau = S \Sigma^{-1}$,
and its AR-quiver is $\bZ D_n$, see Happel~\cite{Happel88}.
The effect of $\Sigma$ on the AR-quiver is given by $\tau^{-n+1}$ if $n$ is
even and by $\tau^{-n+1} \phi$ if $n$ is odd.

Consider the auto-equivalence $F = S^{1-m(e-1)} \Sigma^{m(e-1)-2}$ on
$\cD^b(\modf KD_{me})$ and let $\cC_{m,e} = \cD^b(\modf KD_{me})/F$
be the corresponding orbit category. 
In order to show that $\cC_{m,e}$ is a triangulated 2-Calabi-Yau
category, we use Proposition~\ref{p:CY} and distinguish two cases.
If $m$ is even, then, in the notation of Proposition~\ref{p:CY},
\begin{align*}
(d_1,e_1) = (me-2, me-1), && (d_2,e_2) = (2-m(e-1), 1-m(e-1))
\end{align*}
hence $d_1$, $d_2$ are both even and $e_1 d_2 - e_2 d_1 = m \neq 0$.
Therefore Proposition~\ref{p:CY} applies and the claim follows
since $(2,1) = (e-1)(d_1,e_1) + e(d_2,e_2) \in L$.

If $m$ is odd, then by our assumption $e$ is odd and
\begin{align*}
(d_1,e_1) = (2me-4, 2me-2), && (d_2,e_2) = (2-m(e-1), 1-m(e-1))
\end{align*}
so again $d_1$, $d_2$ are both even and $e_1 d_2 \neq e_2 d_1$.
Therefore Proposition~\ref{p:CY} applies and the claim follows from
$(2,1)=\frac{e-1}{2}(d_1,e_1) + e(d_2,e_2) \in L$.

The effect of $F$ on the AR-quiver of $\cD^b(\modf KD_{me})$
is given by that of 
\[
\tau^{1-m(e-1)} \Sigma^{-1} = \tau^{1-m(e-1)} \tau^{me-1} \phi^{me \bmod 2}
= \tau^m \phi^{m \bmod 2} = (\tau \phi)^m
\]
(using our assumption that $m(e-1)$ is even),
so by~\cite[Prop.~1.2]{BMRRT06} and~\cite[Prop.~1.3]{BMRRT06}
the category $\cC_{m,e}$ is Krull-Schmidt and the shape of its AR-quiver is
$\bZ D_{me}/\langle (\tau \phi)^m \rangle$.

Choose a vertex $v$ with $\phi(v) \neq v$.
Since $\cC$ is 2-CY, the suspension of any object is isomorphic to its
AR-translate. Using this, one can verify from
the mesh relations in $\bZ D_{me}$ that
the sum of the $m$ indecomposables corresponding to the vertices
\[
v, (\phi \tau)v, (\phi \tau)^2 v \ldots, (\phi \tau)^{m-1} v
\]
is a cluster-tilting object in $\cC_{m,e}$ whose endomorphism algebra
is isomorphic to $\gL_{m,e}$. An example for $m=4$ and $e=2$ is
shown in Figure~\ref{fig:C42}. Similar pictures appear in \S2.1
and \S2.2 of~\cite{Ringel08}.

\begin{figure}
\[
\xymatrix@=1pc{
& {\circ} \ar[dr] &&
{\bullet} \ar[dr] && {\circ} \ar[dr] &&
{\bullet} \ar[dr] & {} \\
{\circ} \ar[ur] \ar[r] \ar[dr] \ar@{--}[u] &
{\bullet} \ar[r] & {\circ} \ar[ur] \ar[r] \ar[dr] &
{\circ} \ar[r] & {\circ} \ar[ur] \ar[r] \ar[dr] &
{\bullet} \ar[r] & {\circ} \ar[ur] \ar[r] \ar[dr] &
{\circ} \ar[r] & {\circ} \ar@{--}[u] \\
& {\circ} \ar[ur] \ar[dr] &&
{\circ} \ar[ur] \ar[dr] && {\circ} \ar[ur] \ar[dr] &&
{\circ} \ar[ur] \ar[dr] \\
{\circ} \ar[ur] \ar[dr] \ar@{--}[uu] && {\circ} \ar[ur] \ar[dr] &&
{\circ} \ar[ur] \ar[dr] && {\circ} \ar[ur] \ar[dr] &&
{\circ} \ar@{--}[uu]
\\
& {\circ} \ar[ur] \ar[dr] &&
{\circ} \ar[ur] \ar[dr] && {\circ} \ar[ur] \ar[dr] &&
{\circ} \ar[ur] \ar[dr] \\
{\circ} \ar[ur] \ar[dr] \ar@{--}[uu] \ar@{--}[d]
&& {\circ} \ar[ur] \ar[dr] &&
{\circ} \ar[ur] \ar[dr] && {\circ} \ar[ur] \ar[dr] &&
{\circ} \ar@{--}[uu] \ar@{--}[d] 
\\
& {\circ} \ar[ur] &&
{\circ} \ar[ur] && {\circ} \ar[ur] &&
{\circ} \ar[ur] & {}
}
\]
\caption{The AR-quiver of $\cC_{4,2}$, where the left and right columns have
to be identified along the dashed lines.
The sum of the 4 indecomposables marked in $\bullet$
is a cluster-tilting object with endomorphism algebra $\gL_{4,2}$.}
\label{fig:C42}
\end{figure}
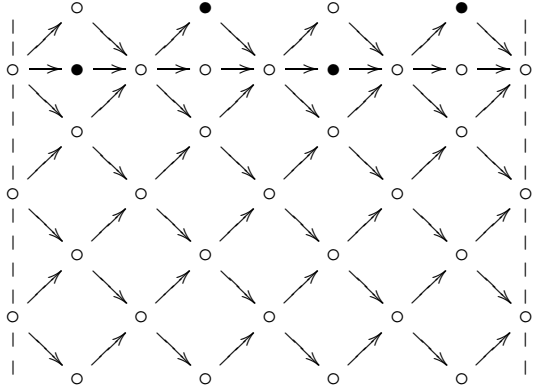

\subsection{Proof of Proposition~\protect{\ref{p:CY}} and its corollary}

Let $D=d_1 e_2 - d_2 e_1$. We have
\[
F^{e_1} = S^{e_1 e_2} \Sigma^{-e_1 d_2} \simeq \Sigma^{d_1 e_2 - e_1 d_2} = \Sigma^D ,
\]
hence if $D=0$ then $F^{e_1}$ is isomorphic to the identity functor
and the endomorphism algebra $\End_{\cC}(X) = \bigoplus_{n \in \bZ}
\Hom_{\dcH}(X, F^n X)$ is infinite-dimensional
for any $X \neq 0$.

Conversely, assume that $D \neq 0$ and let $X, Y \in \dcH$. Since $\cH$
has finite global dimension, we have
\[
\Hom_{\dcH}(X, \Sigma^m F^s Y) = 0
\]
for all $0 \leq s < |e_1|$ and $|m| \gg 0$.
Let $n \in \bZ$ and write it as $n=e_1 m + s$ with $0 \leq s < |e_1|$.
Then
\[
\Hom_{\dcH}(X, F^n Y) = \Hom_{\dcH}(X, F^{e_1 m + s} Y) =
\Hom_{\dcH}(X, \Sigma^{Dm} F^s Y)
\]
vanishes when $|n| \gg 0$ and therefore
$\Hom_\cC(X,Y) = \bigoplus_{n \in \bZ} \Hom_{\dcH}(X, F^n Y)$ is
finite-dimensional.
This shows the equivalence of the conditions in the proposition.

From now on assume that $D \neq 0$. In order to show that the 
orbit category $\cC$ is triangulated we use a result of
Keller~\cite[Theorem~1]{Keller05}. We have to verify conditions~(2)
and~(3) in that theorem. Indeed, condition~(2) holds since if $U \in \cH$,
then $F^{e_1} U \simeq \Sigma^D U$ hence only finitely many objects
$F^i U$ ($i \in \bZ$) could lie in $\cH$.
Moreover, since $\cH$ is hereditary, any indecomposable of $\dcH$
is of the form $\Sigma^n U$ for an indecomposable $U$ of $\cH$ and
some $n \in \bZ$. Writing $n=Dq+t$ with $0 \leq t < |D|$ we see that
$F^{-e_1 q}(\Sigma^n U) \simeq \Sigma^t U$ hence condition~(3) is 
satisfied with the integer $|D|$.

In order to show that $\cC$ is $(d,e)$-CY for any $(d,e) \in L$
we use a result of Dugas~\cite[\S9]{Dugas12}.
The suspension $\bar{\Sigma}$ and the Serre functor $\bar{S}$ on
$\cC$ are induced by $\Sigma$ and $S$, respectively, and
by our assumption $d_2(e_2-d_2)$ is even. We can 
therefore use Theorem~9.5 of~\cite{Dugas12} to deduce that
$\bar{S}^{e_2} \simeq \bar{\Sigma}^{d_2}$.
Now rewrite $F$ as
\[
F = S^{e_2-e_1} S^{e_1} \Sigma^{-d_2} \simeq S^{e_2-e_1} \Sigma^{d_1-d_2}
\]
and since $(d_1-d_2)(d_1-d_2+e_2-e_1)$ is even, by the same theorem
we deduce that $\bar{S}^{e_2-e_1} \simeq \bar{\Sigma}^{d_2-d_1}$.
Since $L$ is generated
by $(d_2,e_2)$ and $(d_2-d_1,e_2-e_1)$, we conclude that
$\cC$ is $(d,e)$-CY for any $(d,e) \in L$.

Note that in any case, even if we do not assume any parity restrictions
on $d_1$, $d_2$, $e_1$ and $e_2$, Theorem~9.5 of~\cite{Dugas12}
implies that $\bar{S}^{2e_2} \simeq \bar{\Sigma}^{2d_2}$ and
$\bar{S}^{2(e_2-e_1)} \simeq \bar{\Sigma}^{2(d_2-d_1)}$, hence
$\cC$ is $(d,e)$-CY for any $(d,e) \in 2L$.

The corollary follows from the following observation. We have
$L \otimes_{\bZ} \bQ = \bQ (d_1,e_1) + \bQ (d_2,e_2) = \bQ^2$ since
$\rank L = 2$. Therefore for any $r \in \bQ$ there are rationals
$s_1, s_2 \in \bQ$ such that
$(r,1) = s_1(d_1,e_1) + s_2(d_2,e_2)$. Multiplying by a common
denominator of $s_1$ and $s_2$ we deduce that $(nr, n) \in L$ for some
integer $n \neq 0$.

\subsection{A cluster category of type $G_2$}
\label{sec:G2}

There are several approaches to categorify cluster algebras corresponding
to non simply-laced Dynkin diagrams, see for example the work by
Demonet~\cite{Demonet11}.

We proceed as in the proof of Proposition~\ref{p:cat}, starting with
a Dynkin quiver of type $E_8$.
Here, $\cD^b(\modf KE_8)$ has translation functor $\Sigma$ and Serre functor
$S$ which are related by $S^{15} \simeq \Sigma^{14}$.
Consider the auto-equivalence $F = S^4 \Sigma^{-4}$ and let
$\cC = \cD^b(\modf KE_8)/F$ be the corresponding orbit category.
From Proposition~\ref{p:CY} we deduce that $\cC$ is a triangulated 2-CY
category, as we have $(2,1)=4 \cdot (4,4)-(14,15)$.
The effect of $F$ on the AR-quiver of $\cD^b(\modf KE_8)$, which is $\bZ E_8$,
is given by $\tau^4$, hence by~\cite[Prop.~1.3]{BMRRT06} the AR-quiver of
$\cC$ has the shape $\bZ E_8/\langle \tau^4 \rangle$.

In order to determine the cluster-tilting objects in $\cC$, we make the
following observations which could be verified on a computer, and refer to
Figure~\ref{fig:E8} for details. We denote by $\bSigma$ the translation on
$\cC$ and observe that $\bSigma^4$ is the identity on $\cC$ since
$(4,0) = 15 \cdot (4,4) - 4 \cdot (14,15)$.
\begin{enumerate}
\renewcommand{\labelenumi}{\theenumi.}
\item
The only indecomposable objects $Z$ which are rigid (i.e.\
$\Hom_{\cC}(Z,\bSigma Z)=0$)
are of the form $\bSigma^i X$ or $\bSigma^i Y$ ($0 \leq i < 4$),
where $X$ and $Y$ are as in Figure~\ref{fig:E8}.

\item
The indecomposables $Z$ such that $\Hom_{\cC}(X,Z)=0$ are shown in 
Fig.~\ref{fig:E8}(a).

\item
The indecomposables $Z$ such that $\Hom_{\cC}(Y,Z)=0$ are shown in
Fig.~\ref{fig:E8}(b).
\end{enumerate}

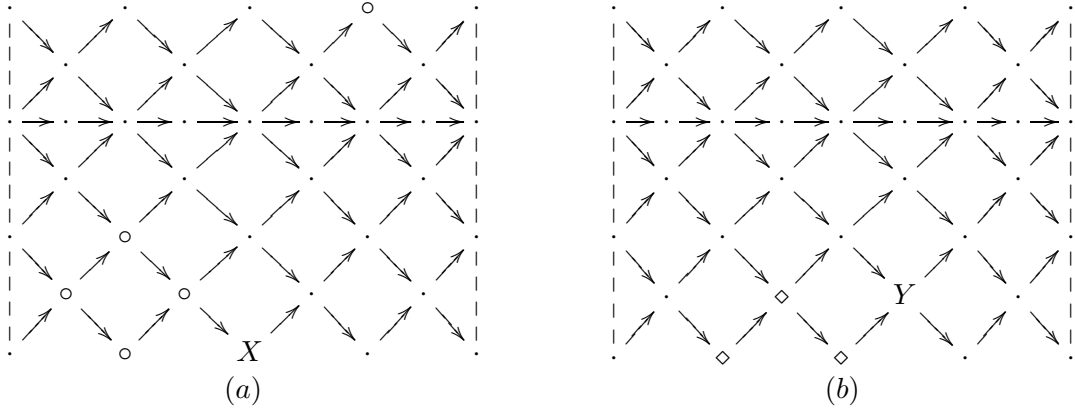
\begin{figure}
\[
\begin{array}{ccc}
\xymatrix@=0.85pc{
{\cdot} \ar[dr] && {\cdot} \ar[dr] &&
{\cdot} \ar[dr] && {\circ} \ar[dr] &&
{\cdot} \\
& {\cdot} \ar[ur] \ar[dr] && {\cdot} \ar[ur] \ar[dr] &&
{\cdot} \ar[ur] \ar[dr] && {\cdot} \ar[ur] \ar[dr] \\
{\cdot} \ar@{--}[uu] \ar[ur] \ar[r] \ar[dr] & {\cdot} \ar[r] &
{\cdot} \ar[ur] \ar[r] \ar[dr] & {\cdot} \ar[r] &
{\cdot} \ar[ur] \ar[r] \ar[dr] & {\cdot} \ar[r] &
{\cdot} \ar[ur] \ar[r] \ar[dr] & {\cdot} \ar[r] &
{\cdot} \ar@{--}[uu] \\
& {\cdot} \ar[ur] \ar[dr] && {\cdot} \ar[ur] \ar[dr] &&
{\cdot} \ar[ur] \ar[dr] && {\cdot} \ar[ur] \ar[dr] \\
{\cdot} \ar@{--}[uu] \ar[ur] \ar[dr] && {\circ} \ar[ur] \ar[dr] &&
{\cdot} \ar[ur] \ar[dr] && {\cdot} \ar[ur] \ar[dr] &&
{\cdot} \ar@{--}[uu] \\
& {\circ} \ar[ur] \ar[dr] && {\circ} \ar[ur] \ar[dr] &&
{\cdot} \ar[ur] \ar[dr] && {\cdot} \ar[ur] \ar[dr] \\
{\cdot} \ar@{--}[uu] \ar[ur] && {\circ} \ar[ur] &&
{X} \ar[ur] && {\cdot} \ar[ur] &&
{\cdot} \ar@{--}[uu]
}
& \qquad &
\xymatrix@=0.85pc{
{\cdot} \ar[dr] && {\cdot} \ar[dr] &&
{\cdot} \ar[dr] && {\cdot} \ar[dr] &&
{\cdot} \\
& {\cdot} \ar[ur] \ar[dr] && {\cdot} \ar[ur] \ar[dr] &&
{\cdot} \ar[ur] \ar[dr] && {\cdot} \ar[ur] \ar[dr] \\
{\cdot} \ar@{--}[uu] \ar[ur] \ar[r] \ar[dr] & {\cdot} \ar[r] &
{\cdot} \ar[ur] \ar[r] \ar[dr] & {\cdot} \ar[r] &
{\cdot} \ar[ur] \ar[r] \ar[dr] & {\cdot} \ar[r] &
{\cdot} \ar[ur] \ar[r] \ar[dr] & {\cdot} \ar[r] &
{\cdot} \ar@{--}[uu] \\
& {\cdot} \ar[ur] \ar[dr] && {\cdot} \ar[ur] \ar[dr] &&
{\cdot} \ar[ur] \ar[dr] && {\cdot} \ar[ur] \ar[dr] \\
{\cdot} \ar@{--}[uu] \ar[ur] \ar[dr] && {\cdot} \ar[ur] \ar[dr] &&
{\cdot} \ar[ur] \ar[dr] && {\cdot} \ar[ur] \ar[dr] &&
{\cdot} \ar@{--}[uu] \\
& {\cdot} \ar[ur] \ar[dr] && {\diamond} \ar[ur] \ar[dr] &&
{Y} \ar[ur] \ar[dr] && {\cdot} \ar[ur] \ar[dr] \\
{\cdot} \ar@{--}[uu] \ar[ur] && {\diamond} \ar[ur] &&
{\diamond} \ar[ur] && {\cdot} \ar[ur] &&
{\cdot} \ar@{--}[uu]
}
\\
(a) & & (b)
\end{array}
\]
\caption{Each picture shows the AR-quiver of $\cC$, where the left and right
columns have to be identified along the dashed lines.
In~(a) we mark by $\circ$ the indecomposables $Z$ with $\Hom_{\cC}(X,Z)=0$,
whereas in~(b) we mark by $\diamond$ those $Z$ such that $\Hom_{\cC}(Y,Z)=0$.}
\label{fig:E8}
\end{figure}

It follows that the cluster-tilting objects in $\cC$ are
$\bSigma^i X \oplus \bSigma^j Y$, where $0 \leq i < 4$ and $j \in \{i, i+1\}$,
their endomorphism algebras are given by the quivers
\begin{equation} \tag{$\star$} \label{e:CTG2}
\xymatrix{
{\bullet} \ar@{-}[r]^{\alpha} & {\bullet} \ar@(ur,dr)[]^{\beta}
}
\end{equation}
(where the edge $\alpha$ can be oriented arbitrarily) with the relation
$\beta^3=0$, and their exchange graph is an octagon as shown in
Figure~\ref{fig:exchange}.

The exchange graph of the cluster algebra of type $G_2$ is also an
octagon (see e.g.\ the last example in~\cite[\S2]{Keller10})
and the connection is explained by the following observation.
It is possible to compute the cluster character in the sense of
Palu~\cite{Palu08} corresponding to the cluster-tilting object
$X \oplus \bSigma Y$.
With this choice, the AR-quiver of the resulting 2-CY-tilted algebra
is the one shown in~\cite[Fig.~19]{Gabriel80}, so we can use the
dimension vectors listed there to aid in the calculations. More calculations
are simplified by the properties of a cluster character. The resulting
values on the rigid indecomposable objects of $\cC$ are shown 
in Figure~\ref{fig:variables},
so there is a bijection compatible with mutations
between the rigid indecomposable objects of $\cC$
and the cluster variables in the cluster algebra of type $G_2$ such that
(basic) cluster-tilting objects in $\cC$ correspond to the clusters.

\begin{figure}
\[
\xymatrix@!0{
&&& {X \oplus Y} \ar@{-}[drr] \\
& {\bSigma^3 X \oplus Y} \ar@{-}[urr]
&&&& {X \oplus \bSigma Y} \ar@{-}[ddr] \\ \\
{\bSigma^3 X \oplus \bSigma^3 Y} \ar@{-}[uur]
&&&&&& {\bSigma X \oplus \bSigma Y} \ar@{-}[ddl] \\ \\
& {\bSigma^2 X \oplus \bSigma^3 Y} \ar@{-}[uul]
&&&& {\bSigma X \oplus \bSigma^2 Y} \ar@{-}[dll] \\
&&& {\bSigma^2 X \oplus \bSigma^2 Y} \ar@{-}[ull]
}
\]
\caption{The exchange graph of the cluster-tilting objects in $\cC$.}
\label{fig:exchange}
\end{figure}
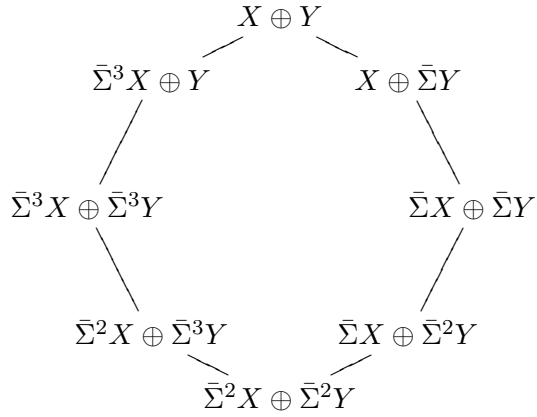

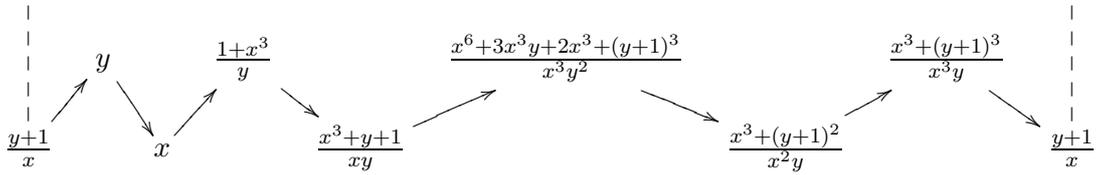
\begin{figure}
\[
\xymatrix@=0.8pc{
{} &&&&&&&& {} \\
& {y} \ar[dr] && {\frac{1+x^3}{y}} \ar[dr] 
&& {\frac{x^6+3x^3y+2x^3+(y+1)^3}{x^3y^2}} \ar[dr]
&& {\frac{x^3+(y+1)^3}{x^3y}} \ar[dr] \\
{\frac{y+1}{x}} \ar@{--}[uu] \ar[ur] && {x} \ar[ur]
&& {\frac{x^3+y+1}{xy}} \ar[ur]
&& {\frac{x^3+(y+1)^2}{x^2y}} \ar[ur] 
&& {\frac{y+1}{x}} \ar@{--}[uu]
}
\]
\caption{The values of a cluster character on the rigid indecomposable
objects of $\cC$ are the cluster variables of the cluster algebra
of type $G_2$.}
\label{fig:variables}
\end{figure}

Note that $F^4 = S^{16} \Sigma^{-16} \simeq S \Sigma^{-2}$, hence as in the
cases treated by Bertani-{\O}kland and Oppermann
in~\cite{BertaniOppermann11}, there is a covering functor from
the cluster category of type $E_8$ to $\cC$. 
Each of the algebras depicted in~\eqref{e:CTG2} is obtained from
a corresponding cluster-tilted algebra of type $E_8$ whose quiver is shown
below (where the edges labeled $\alpha$ should have the same orientation
either pointing inwards or outwards with respect to the inner square)
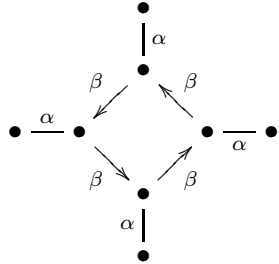
\[
\xymatrix@=1pc{
&& {\bullet} \ar@{-}[d]^{\alpha} \\
&& {\bullet} \ar[dl]_{\beta} \\
{\bullet} \ar@{-}[r]^{\alpha} & {\bullet} \ar[dr]_{\beta}
&& {\bullet} \ar[ul]_{\beta} & {\bullet} \ar@{-}[l]^{\alpha} \\
&& {\bullet} \ar[ur]_{\beta} \\
&& {\bullet} \ar@{-}[u]^{\alpha}
}
\]
by identifying all the arrows with the same label
(equivalently, by dividing modulo the action of the cyclic group
$\bZ/4\bZ$ on the quiver by rotations).
These algebras appear as items 1562 and 1574 in the lists supplementing
the paper~\cite{BHL13}.

\bibliographystyle{amsplain}
\bibliography{jac2cy}

\end{document}